\documentclass[hidelinks]{article}
\usepackage{graphicx} 
\usepackage[a4paper, margin=2.5cm]{geometry}
\usepackage{hyperref}
\usepackage{amsthm}
\usepackage{thm-restate}
\usepackage{geometry}
\usepackage{amssymb}
\usepackage{amsmath}
\usepackage{cleveref}
\usepackage{comment}
\usepackage{xcolor}
\usepackage[shortlabels]{enumitem}
\usepackage{listings}
\usepackage{blkarray}
\usepackage{lmodern}
\usepackage[english]{babel}
\usepackage{float}
\usepackage{tikz}

\newtheorem{theorem}{Theorem}[section]
\newtheorem{lemma}[theorem]{Lemma}

\theoremstyle{definition}

\title{Counting sparse induced subgraphs in locally dense graphs}

\author{
Rajko Nenadov\thanks{School of Computer Science, University of Auckland, New Zealand. Email: \texttt{rajko.nenadov@auckland.ac.nz}. Research supported by the Marsden Fund of the Royal Society of New Zealand.}
}
\date{}

\begin{document}

\maketitle

\begin{abstract}
An $n$-vertex graph $G$  is locally dense if every induced subgraph of size larger than $\zeta n$ has density at least $d > 0$, for some parameters $\zeta, d > 0$. We show that the number of induced subgraphs of $G$ with $m$ vertices and maximum degree significantly smaller than $dm$ is roughly $\binom{\zeta n}{m}$, for $m \ll \zeta n$ which is not too small. This generalises a result of Kohayakawa, Lee, R\"odl, and  Samotij on the number of independent sets in locally dense graphs. 
As an application, we slightly improve a result of Balogh, Chen, and Luo on the generalised Erd\H{o}s-Rogers function for graphs with small extremal number. 
\end{abstract}

\section{Introduction}

Many problems in combinatorics can be phrased in terms of independent sets in graphs and hypergraphs. For example, a classical result of Szemer\'edi \cite{szemeredi75ap} on arithmetic progressions in dense subsets of integers states that if we form a $k$-uniform hypergraph on the vertex set $\{1, \ldots, n\}$ by putting a hyperedge on top of each $k$-term arithmetic progression, then the largest independent set in this hypergraph is of size $o(n)$. In general, problems where one wishes to show that sets of certain size necessarily contain a desired configuration can often be formulated as problems of estimating largest independent sets in (hyper)graphs where edges encode such configurations. 

In many applications, such as transference of extremal results to random structures \cite{balogh15containers,saxton15containers} and the recent breakthrough on the off-diagonal Ramsey number $R(4, t)$ \cite{mattheus24ramsey}, it is also useful to know not just the size of a largest independent set but also the number of independent sets of certain size. Building on the  work of Kleitman and Winston \cite{kleitman82cycles} and Sapozhenko \cite{sapozhenko99containers}, this line of research has culminated with the so-called \emph{hypergraph containers} of Balogh, Morris, and Samotij \cite{balogh15containers} and Saxton and Thomason \cite{saxton15containers} and has lead to many important discoveries. For a thorough survey, we refer the reader to \cite{balogh18survey,samotij15survey}. In this note, we are interested in counting sets which are sparse but not necessarily independent.

We say that an $n$-vertex graph $G$ is \emph{$(\zeta, d)$-dense}, for some $\zeta, d > 0$, if for every $S \subseteq V(G)$ of size $|S| \ge \zeta n$ we have $e(G[S]) \ge d |S|^2/2$. The following is our main technical contribution.

\begin{lemma} \label{lemma:count}
    Let $G$ be a $(\zeta, d)$-dense graph with $n$ vertices, for some $\zeta, d > 0$ which may depend on $n$. For any integer $D \ge 1$ and $s \ge f := (4D/d) \log (1/\zeta)$, there are at most 
    $$
        \binom{n}{f} \binom{\zeta n}{s - f}
    $$
    subsets $U \subseteq V(G)$ of size $|U| = s$ such that $\Delta(G[U]) < D$.
\end{lemma}

Note that $D$ may be a function of $n$. Lemma \ref{lemma:count} extends a result of Kohayakawa, Lee, R\"odl, and Samotij on independent sets in locally dense graphs, which corresponds to $D = 1$. 


Next, we describe an application of Lemma \ref{lemma:count}.

\subsection{Generalised Erd\H{o}s-Rogers function}

Erd\H{o}s and Rogers \cite{erdos62rogers} considered the following generalisation of the off-diagonal Ramsey problem: For positive integers $2 \le s < r$ and $n$, let $f_{s,r}(n)$ denote the largest $m$ such that every $K_r$-free graph on $n$ vertices contains a $K_s$-free induced subgraph on $m$ vertices. The function $f_{s,r}(n)$ is known as the Erd\H{o}s-Rogers function and has been studied extensively (for some recent progress, see \cite{dudek14erdos,gowers20improved,janzer2023improved,mubayi2024improved,wolfovits13triangle}; a thorough reference list can be found in \cite{gishboliner2024induced}). Recently, Balogh, Chen, and Luo \cite{balogh2024maximum} and Mubayi and Verstraete \cite{mubayi2024erd} initiated the systematic study of a generalised Erd\H{o}s-Rogers function: Given graphs $F$ and $H$, let $f_{F,H}(n)$ denote the largest $m$ such that every $H$-free graph with $n$ vertices contains an $F$-free induced subgraph on $m$ vertices. The generalised Erd\H{o}s-Rogers function has been further studied by Gishboliner, Janzer, and Sudakov \cite{gishboliner2024induced}. All of the previous work focuses on the case where $H$ is a complete graph.

In the language of Erd\H{o}s-Rogers functions, the recent breakthrough of Mattheus and Verstraete \cite{mattheus24ramsey} on the off-diagonal Ramsey number $R(4,t)$ implies $f_{K_2,K_4}(n) = O \left( n^{1/3} (\log n)^{4/3} \right)$. Balogh, Chen, and Luo \cite{balogh2024maximum} showed that for any graph $F$ with $\mathrm{ex}(F, m) = O(m^{1 + \alpha})$ , for some $\alpha \in [0, 1/2)$, we have
\begin{equation} \label{eq:balogh}
    f_{F, K_4} = O \left( n^{\frac{1}{3 - 2\alpha}} (\log n)^{\frac{6}{3 - 2\alpha}} \right).
\end{equation}
Recall that $\mathrm{ex}(F, m)$ denotes the largest number of edges in an $F$-free graph with $m$ vertices. Using Lemma \ref{lemma:count}, we slightly improve the logarithmic factor in \eqref{eq:balogh}.

\begin{theorem} \label{thm:main}
    For every graph $F$ with $\mathrm{ex}(F, m) = O\left( m^{1 + \alpha} \right)$, where $\alpha \in [0, 1/2)$, we have
    $$
        f_{F, K_4}(n) = O\left( n^{\frac{1}{3 - 2\alpha}} (\log n)^{\frac{4}{3 - 2\alpha}} \right).
    $$
\end{theorem}

The statement of Theorem \ref{thm:main} actually holds for all $\alpha \in [0, 1]$, however for $\alpha \ge 1/2$ it is rather uninformative. For $\alpha \ge 1/2$ it gives a bound of order at least $n^{1/2} (\log n)^2$, and it was proved by Mubayi and Verstraete \cite{mubayi2024improved} that $f_{F,K_4}(n) = O\left( n^{1/2} \log n \right)$ for every $K_4$-free graph $F$. Even more recently, Gishboliner, Janzer, and Sudakov \cite{gishboliner2024induced} showed that if $F$ is triangle-free, then $f_{F, K_4}(n) = O\left(n^{1/2 - c_F}\right)$, where $c_F > 0$ is some constant which depends on $F$.

\section{Proof of the counting lemma}

The following fingerprint--container lemma is the key ingredient in the proof of Lemma \ref{lemma:count}. 

\begin{lemma} \label{lemma:container}
    Suppose $G$ is a graph with $n$ vertices such that for every subset $S \subseteq V(G)$ of size $|S| \ge k$ we have $\Delta(G[S]) \ge \Delta$. Then for every $U \subseteq V(G)$ such that $\Delta(G[U]) < D$ and $|U| \ge \lceil 2 n D / \Delta \rceil =: \ell$, there exists a subset $F \subseteq U$ of size $|F| = \ell$ and $C = C(F, D, \Delta) \subseteq V(G)$ of size $|C| < k$, such that $U \subseteq C \cup F$. Importantly, given $D$ and $\Delta$, the set $C$ can be constructed from any set $U'$ such that $F \subseteq U' \subseteq U$.
\end{lemma}
\begin{proof}
    Consider a subset $U \subseteq V(G)$, $|U| \ge \ell$, such that $\Delta(G[U]) < D$. Set $R_0 := \emptyset$ and $F := \emptyset$, and repeat the following for $i = 1, \ldots, \ell$:
    \begin{itemize}
        \item Pick a vertex $v_i \in U \setminus F$ with the largest degree in $G \setminus R_{i-1}$, tie-breaking according to some fixed ordering of the vertices in $V(G)$.
        \item Add $v_i$ to $F$, and set $R_i := \{w \in V(G) \colon \deg(w, F) \ge D \}$.
    \end{itemize}
    By the assumption on $\Delta(G[U])$, we maintain $U \subseteq V(G) \setminus R_i$ throughout the process. Note that the algorithm produces the same sets $F$ and $R_\ell$ if run with $U'$ instead of $U$, where $F \subseteq U' \subseteq U$.
    
    Set $Z = V(G) \setminus R_\ell$. We claim that the last vertex $v_\ell$ added to $F$ satisfies $\deg(v_\ell, Z) < \Delta$. Suppose, towards a contradiction, that $\deg(v_\ell, Z) \ge \Delta$. Then $\deg(v_i, V(G) \setminus R_{i-1}) \ge \Delta$ for every $i \in \{1, \ldots, \ell\}$. Create an auxiliary subgraph $G' \subseteq G$, where for each $i \in \{1, \ldots, \ell\}$ we add all the edges incident to $v_i$ in $G \setminus R_{i-1}$. By the definition of $R_i$, the maximum degree of vertices in $V(G) \setminus F$ in $G'$ is $D$. As the degree of each $v_i \in F$ in $G'$ is at least $\Delta$ and $e(G'[F]) \le e(G[F]) < |F|D$, we conclude $e_{G'}(F, V(G) \setminus F) > |F| \Delta - |F| D \ge n D$. However, this is not possible as each vertex in $V(G) \setminus F$ has degree at most $D$ in $G'$, thus $e_{G'}(F, V(G) \setminus F) < n D$.
   
    Having established $\deg(v, Z) < \Delta$ for every $v \in U \setminus F$, we set $C := \{w \in Z \colon \deg(w, Z) < \Delta \}$. By the assumption of the lemma, we conclude $|C| < k$.
\end{proof}

Lemma \ref{lemma:count} follows by iterating Lemma \ref{lemma:container}.

\begin{proof}[Proof of Lemma \ref{lemma:count}]
    We show that for every $U \subseteq V(G)$, $|U| \ge f$, such that $\Delta(G[U]) < D$, there exists a subset $F \subseteq U$ of size $|F| = f$ and a subset $C = C(F) \subseteq V(G)$ of size $|C| < \zeta n$ such that $U \setminus F \subseteq C$ and $C$ depends only on $F$. This implies that we can enumerate all such sets $U$ of size $s \ge f$ by going over all subsets $F \subseteq V(G)$ of size $f$, and for each $F$ choosing a subset of $C(F)$ of size $s - f$. This yields the bound stated in the lemma.

    Consider one such subset $U \subseteq V(G)$. Set $F= \emptyset$ and $C = V(G)$. As long as $|C| \ge \zeta n$, apply Lemma \ref{lemma:container} on the induced subgraph $G[C]$ with $k = \max\{\zeta n, |C|/2\}$, $\Delta = |C|d/2$ (follows from the local density assumption for sets of size $k$) and $D$ to obtain $F' \subseteq U \cap C$ and $C' \subseteq C$; set $F := F \cup F'$ and $C := C'$ and proceed to the next iteration. In each step the size of $C$ drops by a factor of at least $2$, thus the process terminates after at most $\log(1/\zeta)$ iterations. The set $F$ increases by a factor of $2 |C| D / \Delta = 4 D / d$ in each iteration, thus we terminate with $F$ being of size at most $f$. In case $|F| < f$, we further add elements from $U \setminus F$ to $F$, in some canonical order, such that $|F| = f$. Note that at the end of the process we have $F\subseteq U \subseteq C \cup F$. By the second property of Lemma \ref{lemma:container}, we can construct the same set $C$ if we repeat the process with $U$ replaced by $F$, implying that $C$ indeed depends only on $F$ and not on the whole $U$.
\end{proof}

\section{Graphs without large sparse subgraphs}

We derive Theorem \ref{thm:main} from the following result.

\begin{theorem} \label{thm:no_sparse}
    For every $t$ and $D$ there exists a  $K_4$-free graph $G$ with
    $$
        \Theta\left( \frac{t^3}{D^2 (\log t)^4} \right)
    $$
    vertices such that every subset $S \subseteq V(G)$ of size $|S| \ge t$ contains at least $D|S|$ edges.
\end{theorem}

Note that Theorem \ref{thm:main} extends the result by Mattheus and Verstraete \cite{mattheus24ramsey} who constructed a $K_4$-free graph with $\Theta\left( t^3 / (\log t)^4\right)$ vertices such that every subset of size $t$ contains an edge. En route to prove Theorem \ref{thm:no_sparse}, we first prove a related result for locally dense graphs. 

\begin{lemma} \label{lemma:small_subsets}
    Let $G$ be a $(\zeta, d)$-dense graph with $n$ vertices, for some $\zeta, d > 0$ which may depend on $n$. Let $D \ge 1$ and set $s = \frac{16 D}{d} \log(1/\zeta) \log n$. Then there exists a set $V' \subseteq V(G)$ of size $|V'| \ge s / (4 e \zeta)$, such that for every $S \subseteq V'$ of size $|S| \ge 2s$ we have $e(G[S]) \ge D|S|$.
\end{lemma}
\begin{proof}
    We prove the lemma using the probabilistic method. Let $\mathcal{B}_s$ denote the family of all subsets $U \subseteq V(G)$ of size $|U| = s$ such that $\Delta(G[U]) < 2D$. Set $f = \frac{8D}{d} \log(1/\zeta)$, and note that $s = 2 f \log n$. By Lemma \ref{lemma:count}, we have
    $$
        |\mathcal{B}_s| \le  \binom{n}{f} \binom{\zeta n}{s - f}.
    $$
    Form $V' \subseteq V(G)$ by taking each $v \in V(G)$ with probability $p = f \log n / (2 e \zeta n)$, independently. Let $\mathcal{E}$ denote the event that there exists $U \in \mathcal{B}_s$ such that $U \subseteq V'$. By union bound, we get
    \begin{align*}    
        \mathbf{Pr}[\mathcal{E}] \le |\mathcal{B}_s| p^s \le \binom{n}{f} \binom{\zeta n}{s - f} p^s &\le n^f \left( \frac{e \zeta n}{s - f} \right)^{s - f} p^s \le 2^{f \log n} \left( \frac{e \zeta n}{f \log n} \right)^{s - f} p^{s-f} \\
        &\le  2^{f \log n} \left( \frac{e \zeta n}{f \log n} p \right)^{s - f} < 2^{f \log n} 2^{-(s-f)} < 1/2.
    \end{align*}
    By Chernoff bound we have $|V'| < s / (4 e \zeta)$ with probability less than $1/2$. Therefore, with positive probability neither of the two bad events happen.

    Consider some set $S \subseteq V'$ of size $|S| \ge 2s$. Since every subsets $U \subseteq S$ of size $|U| = s$ satisfies $\Delta(G[U]) \ge 2D$, we conclude $e(G[S]) \ge D|S|$. 
\end{proof}

Theorem \ref{thm:no_sparse} follows from Lemma \ref{lemma:small_subsets} and the following result of Mattheus and Verstraete \cite{mattheus24ramsey}. We note that this result was instrumental in all recent work on the Erd\H{o}s-Rogers function.

\begin{theorem}[Theorem 3 in \cite{mattheus24ramsey}] \label{thm:graph}
For every prime power $q > 2^{40}$, there exists a $K_4$-free graph $G$ with $N_q = q^2 (q^2 - q + 1)$ vertices which is $(2^{25}/q^2, 1/(128 q))$-dense.
\end{theorem}

\begin{proof}[Proof of Theorem \ref{thm:no_sparse}]
Let $G$ be as given by Theorem \ref{thm:graph} for sufficiently large $q$. By Lemma \ref{lemma:small_subsets}, there exists $V' \subseteq V(G)$ of size
$$
    |V'| = \Omega\left( D q^3 (\log q)^2 \right)
$$
such that every set $S \subseteq V'$ of size $|S| \ge t$ contains at least $D|S|$ edges, where
$$
    t = \Theta\left( D q (\log q)^2 \right).
$$
As $|V'| = \Omega\left( \frac{ t^3}{D^2 (\log t)^4} \right)$, the graph $G[V']$ has the desired property.

\end{proof}

Finally, we deduce Theorem \ref{thm:main} from Theorem \ref{thm:no_sparse}.

\begin{proof}[Proof of Theorem \ref{thm:main}]
    Consider a graph $F$ such that $\mathrm{ex}(F, t) \le Ct^{1 + \alpha}$, for some constants $C > 0$ and $\alpha \in [0, 1/2)$, for every $t$. Let $G$ be the graph given by Theorem \ref{thm:no_sparse} for $D = Ct^{\alpha}$, for some sufficiently large $t$. The graph $G$ has
    $$
        n = \Theta\left( \frac{t^{3 - 2\alpha}}{(\log t)^4} \right)
    $$
    vertices, and every subset $S \subseteq V(G)$ of size $t$ induces at least $Ct^{1 + \alpha}$ edges, thus it contains a copy of $F$. As $t = \Theta\left(n^{1 / (3 - 2\alpha)} (\log n)^{4/(3 - 2\alpha)}\right)$, this proves the statement.
\end{proof}

\paragraph{Acknowledgment.} The author thanks Zach Hunter for pointing out a subtle mistake in the proof of Lemma \ref{lemma:count}.

\bibliographystyle{abbrv}
\bibliography{references}

\end{document}